\documentclass[graybox]{svmult}

\usepackage{mathptmx}       
\usepackage{helvet}         
\usepackage{courier}        
\usepackage{type1cm}        
\usepackage{makeidx}         
\usepackage{graphicx}        
\usepackage{multicol}        
\usepackage[bottom]{footmisc}

\usepackage{latexsym}
\usepackage{indentfirst}
\usepackage{amsmath}
\usepackage{amssymb}

\newcommand{\cu}{{\mathcal U}}
\newcommand{\cb}{{\mathcal B}}

\makeindex             

\begin{document}

\title*{Invariant, super and quasi-martingale functions of a Markov process}
\author{Lucian Beznea and Iulian C\^impean}
\institute{Lucian Beznea \at Simion Stoilow Institute of Mathematics  of the Romanian Academy,
Research unit No. 2, P.O. Box  1-764, RO-014700 Bucharest, Romania,
University of Bucharest, Faculty of Mathematics and Computer Science, and  Centre Francophone en Math\'ematique de Bucarest,\\
\email{lucian.beznea@imar.ro}
\and Iulian C\^impean \at Simion Stoilow Institute of Mathematics  of the Romanian Academy,
Research unit No. 2, P.O. Box  1-764, RO-014700 Bucharest, Romania, \email{iulian.cimpean@imar.ro}}

\maketitle

\vspace*{-32mm}

\noindent
{\it Dedicated to Michael R\"ockner  on the occasion of his sixtieth   birthday}\\[4mm]

\abstract*{Each chapter should be preceded by an abstract (10--15 lines long) that summarizes the content. 
The abstract will appear \textit{online} at \url{www.SpringerLink.com} and be available with unrestricted access. 
This allows unregistered users to read the abstract as a teaser for the complete chapter. 
As a general rule the abstracts will not appear in the printed version of your book unless it is the style of 
your particular book or that of the series to which your book belongs.
Please use the 'starred' version of the new Springer \texttt{abstract} command for typesetting the text of the online abstracts 
(cf. source file of this chapter template \texttt{abstract}) and include them with the source files of your manuscript. 
Use the plain \texttt{abstract} command if the abstract is also to appear in the printed version of the book.}

\abstract{ We identify the linear space spanned by the real-valued excessive functions of a Markov process 
with the set of those functions which are quasimartingales when we compose them with the process. 
Applications to semi-Dirichlet forms are given.
We provide a unifying result which clarifies the relations between harmonic, 
co-harmonic, invariant, co-invariant, martingale and co-martingale functions, 
showing that in the conservative case they are all the same.
Finally, using the co-excessive functions, we  present a two-step approach to the existence of invariant probability measures.
}

\vspace{0.5cm}




\section{Introduction}
\label{sec:1}

Let $E$ be a Lusin topological space endowed with the Borel $\sigma$-algebra $\mathcal{B}$ and 
$X = (\Omega, \mathcal{F}, \mathcal{F}_t, X_t, \mathbb{P}^x, \zeta)$ be a right Markov process with state space $E$, 
transition function $(P_t)_{t \geq 0}$:  
$P_t u(x) = \mathbb{E}^x (u(X_t); t < \zeta)$, $t \geq 0$, $x\in E$.

One of the fundamental connections between potential theory and Markov processes is 
the relation between excessive functions and (right-continuous) supermartingales; see e.g. \cite{Do01}, 
Chapter VI, Section 10, or \cite{LG06}, Proposition 13.7.1 and Theorem 14.7.1. 
Similar results hold for (sub)martingales, and together stand as a keystone at the foundations of the so called probabilistic potential theory. 
For completeness, let us give the precise statement;
a short proof is included in Appendix.

\begin{proposition} \label{prop1}
The following assertions are equivalent for a non-negative real-valued $\mathcal{B}$-measurable function $u$ and $\beta \geq 0$.

i) $(e^{-\beta t}u(X_t))_{t\geq 0}$ is a right continuous $\mathcal{F}_t$-supermartingale w.r.t. $\mathbb{P}^x$ for all $x \in E$.

ii) The function  $u$ is $\beta$-excessive.
\end{proposition}

Our first aim is to show that this connection can be extended to the space of 
differences of excessive functions on the one hand, and to {\it quasimartingales} 
on the other hand (cf. Theorem \ref{thm 2.1}  from Section \ref{sec:2}),  with concrete applications to semi-Dirichlet forms 
(see Theorem \ref{thm 3} below). 

\begin{remark} \label{remark1}
 Recall the following famous characterization from \cite{CiJaPrSh80}: 
 {\it If $u$ is a real-valued $\mathcal{B}$-measurable function then $u(X)$ is an $\mathcal{F}_t$-semimartingale w.r.t. all $\mathbb{P}^x$, $x\in E$ if and only if $u$ is locally the difference of two finite $1$-excessive functions.}
 \end{remark}
 
The main result  from Theorem \ref{thm 2.1}  should be regarded  as an extension of Proposition \ref{prop1} 
and as a refinement of the just mentioned characterization for semimartingales from Remark \ref{remark1}. 
However, we stress out that our result is not a consequence of the two previously known results. 

In Section \ref{sect.3} we focus on a special class of ($0$-)excessive functions called invariant, 
which were studied in the literature from several slightly different perspectives. 
Here, our aim is to provide a unifying result which clarifies the relations between harmonic, 
co-harmonic, invariant, and co-invariant functions, showing that in the Markovian (conservative) case they are all the same. 
The measurable structure of invariant functions is also involved. 
We give the results in terms of $L^p(E,m)$-resolvents of operators, where $m$ is assumed sub-invariant, allowing us to drop the strong continuity assumption.
In addition, we show that when the resolvent is associated to a right process, then the martingale functions and the
co-martingale ones (i.e., martingale w.r.t. to a dual process) also coincide.

The last topic where the existence of (co)excessive functions plays a fundamental role 
is the problem of existence of invariant probability measures 
for a fixed Markovian transition function $(P_t)_{t\geq 0}$ on a general measurable space $(E,\mathcal{B})$. 
Recall that the classical approach is to consider the dual semigroup of $(P_t)_{t\geq 0}$ 
acting on the space of all probabilities $P(E)$ on $E$, and to show that it or its integral means, 
also known as the Krylov-Bogoliubov measures, 
are relatively compact w.r.t. some convenient topology (metric) on $P(E)$ 
(e.g. weak topology, (weighted) total variation norm, Wasserstein metric, etc). 
In essence, there are two kind of conditions which stand behind the success of this approach: 
some (Feller) regularity of the semigroup $(P_t)_{t\geq 0}$ 
(e.g. it maps bounded and continuous (Lipschitz) functions into bounded and continuous (Lipschitz) functions), 
and the existence of some compact (or {\it small}) sets which are infinitely often visited by the process; 
see e.g. \cite{MeTw93a}, \cite{MeTw93b}, \cite{MeTw93c}, \cite{DaZa96}, \cite{LaSz06}, \cite{Ha10}, \cite{KoPeSz10}.
Our last aim is to present (in Section 4)  a result from  \cite{BeCiRo15a}, 
which offers a new (two-step) approach to the existence of invariant measures (see Theorem \ref{thm 2.3} below). 
In very few words, our idea was to first fix a convenient {\it auxilliary} measure $m$ 
(with respect to which each $P_t$ respects classes), and then to look at the dual semigroup of 
$(P_t)_{t\geq 0}$ acting not on measures as before, but on functions. 
In this way we can employ some weak $L^1(m)$-compactness 
results for the dual semigroup in order to produce a non-zero and non-negative co-excessive function.

At this point we would like to mention that most of the announced results, 
which are going to be presented in the next three sections, 
are exposed with details  in
\cite{BeCiRo15}, \cite{BeCiRo15a}, and \cite{BeCi16}. 

The authors had the pleasure to be coauthors of Michael R\"ockner and part of the results presented in this survey paper were obtained jointly.
So, let us conclude this introduction with a

\centerline{"Happy Birthday, Michael!"}

\section{Differences of excessive functions and quasimartingales of Markov processes}  
\label{sec:2}

Recall that the purpose of this section is to study those real-valued measurable functions $u$ having the property that $u(X)$ 
is a {\it $\mathbb{P}^x$-quasimartingale} for all $x \in E$ (in short, "$u(X)$ is a quasimartingale", or "$u$ is a quasimartingale function"). At this point we would like to draw the attention to the fact that in the first part of this section we study quasimartingales with respect to $\mathbb{P}^x$ for all $x \in E$, in particular all the inequalities involved are required to hold pointwise for all $x \in E$.
Later on we shall consider semigroups or resolvents on $L^p$ or Dirichlet spaces with respect to some duality measure, and in these situations we will explicitly mention if the desired properties are required to hold almost everywhere or outside some exceptional sets.

For the reader's convenience, let us briefly present some classic facts about 
quasimartingales in general.
 
\begin{definition} \label{defi 2.1}
Let $(\Omega, \mathcal{F}, \mathcal{F}_t, \mathbb{P})$ 
be a filtered probability space satisfying the usual hypotheses. 
An $\mathcal{F}_t$-adapted, right-continuous integrable process $(Z_t)_{t \geq 0}$ is called $\mathbb{P}$-{\rm quasimartingale} if
$$
{Var}^\mathbb{P}(Z):= \mathop{\sup}\limits_{\tau} \mathbb{E} \{ \mathop{\sum}\limits_{i = 1}^{n} |\mathbb{E}[Z_{t_i} - Z_{t_{i-1}}|\mathcal{F}_{t_{i-1}}]| + |Z_{t_n}|\} < \infty,
$$
where the supremum is taken over all partitions $\tau : 0 = t_0 \leq t_1 \leq \ldots \leq t_n < \infty$.
\end{definition}

Quasimartingales played an important role in the development of the theory 
of semimartingales and stochastic integration, mainly due to M. Rao's theorem according to which any 
quasimartingale has a unique decomposition as a sum of
a local martingale and a predictable process with paths of locally integrable variation. 
Conversely, one can show that any semimartingale with bounded jumps is locally a quasimartingale. 
However, to the best of our knowledge, their analytic or potential theoretic aspects have never been investigated or, maybe, brought out to light, before.

We return now to the frame given by a Markov process. 
Further in this section we deal with a right Markov process $X = (\Omega, \mathcal{F}, \mathcal{F}_t, X_t, \mathbb{P}^x, \zeta)$ 
with state space $E$ and transition function $(P_t)_{t \geq 0}$.
Although we shall not really be concerned with the lifetime formalism, if $X$ has lifetime $\xi$ and cemetery  point $\Delta$, 
then we make the convention $u(\Delta) = 0$ for all functions $u: E \to [-\infty, + \infty]$.

Recall that for $\beta \geq 0$, a $\mathcal{B}$-measurable function 
$f:E \rightarrow [0, \infty]$ is called {\it $\beta$-supermedian} if $P_t^\beta f \leq f$ pointwise on $E$, $t \geq 0$; 
 $(P_t^\beta)_{t \geq 0}$ denotes the  $\beta$-level of the semigroup of kernels $(P_t)_{t \geq 0}$, $P_t^\beta:= e^{-\beta} P_t$.
If $f$ is $\beta$-supermedian and $\lim\limits_{t \to 0} P_t f = f$  point-wise on $E$, then it is called {\it $\beta$-excessive}. 
It is well known that a $\mathcal{B}$-measurable function $f$ is $\beta$-excessive if and only if 
$\alpha U_{\alpha+\beta}f \leq f$, $\alpha >0$,  and $\lim\limits_{\alpha \to \infty} \alpha U_{\alpha}f = f$ point-wise on $E$, 
where $\mathcal{U} = (U_{\alpha})_{\alpha > 0}$ is the resolvent family of the process $X$, 
$U_\alpha := \int_0^\infty e^{-\alpha t} P_t dt$.
The convex cone of all $\beta$-excessive functions is denoted by $E(\mathcal{U}_\beta)$; 
here $\mathcal{U}_\beta$ denotes the $\beta$-level of the resolvent $\mathcal{U}$, $\mathcal{U}_\beta:= (U_{\beta +\alpha})_{\alpha > 0}$; the {\it fine topology} is the coarsest topology on $E$ such that all $\beta$-excessive functions are continuous, for some $\beta > 0$. 
If $\beta = 0$ we  drop the index $\beta$.

Taking into account the strong connection between excessive functions and supermartingales for Markov processes,
 the following characterization of M. Rao was our source of inspiration:
{\it a real-valued process on a filtered probability space $(\Omega, \mathcal{F}, \mathcal{F}_t, \mathbb{P})$ satisfying the usual hypotheses is a quasimartingale 
if and only if it is the difference of two positive right-continuous $\mathcal{F}_t$-supermartingales;}
see e.g. \cite{Pr05}, page 116.

As a first observation, note that if $u(X)$ is a quasimartingale, then the following two conditions for $u$ are necessary:
i) $\mathop{\sup}\limits_{t > 0} P_t|u|< \infty$ and ii) $u$ is finely continuous.
Indeed, since for each $x\in E$ we have that
$\mathop{\sup}\limits_{t} P_t|u|(x) = \mathop{\sup}\limits_{t}\mathbb{E}^x|u(X_t)| \leq {Var}^{\mathbb{P}^x}(u(X)) < \infty$, the first assertion is clear. 
The second one follows by the result from \cite{BlGe68} which is stated in the proof of Proposition \ref{prop1} in the Appendix
at the end of the paper.

For a real-valued function $u$, a partition $\tau$ of $\mathbb{R}^+$, $\tau : 0 = t_0 \leq t_1 \leq \ldots \leq t_n < \infty$, and $\alpha >0$ we set

\centerline{$V^\alpha(u) := \mathop{\sup}\limits_{\tau}V^\alpha_{\tau}(u), \quad V^\alpha_{\tau}(u) := \mathop{\sum}\limits_{i=1}^{n} P^\alpha_{t_{i-1}} |u - P^\alpha_{t_i - t_{i-1}}u| + P^\alpha_{t_n}|u|$,}

\noindent
where the supremum is taken over all finite partitions of $\mathbb{R}_+$. 

A sequence $(\tau_n)_{n \geq 1}$ of finite partitions of $\mathbb{R}_+$ is called 
{\it admissible} if it is increasing, $\mathop{\bigcup}\limits_{k \geq 1}\tau_k$ is dense in $\mathbb{R}_+$, 
and if $r \in \mathop{\bigcup}\limits_{k \geq 1}\tau_k$ then $r + \tau_n \subset \mathop{\bigcup}\limits_{k \geq 1}\tau_k$ for all $n \geq 1$.

We can state now our first result, it  is   a version of Theorem 2.6 from \cite{BeCi16}.

\begin{theorem} \label{thm 2.1}  
Let $u$ be a real-valued $\mathcal{B}$-measurable function and $\beta \geq 0$ such that $P_t|u| < \infty$ for all $t$. 
Then the following assertions are equivalent.

\vspace{0.2cm}
i) $(e^{-\beta t}u(X_t))_{t\geq 0}$ is a $\mathbb{P}^x$-quasimartingale for all $x \in E$. 

\vspace{0.2cm}
ii) $u$ is finely continuous and $\mathop{\sup}\limits_{n}V^\beta_{\tau_n}(u) < \infty$ for one (hence all) admissible sequence of partitions $(\tau_n)_n$.
 
\vspace{0.2cm}
iii) $u$ is a difference of two real-valued $\beta$-excessive functions. 

\end{theorem}

\begin{remark}
The key idea behind the previous result is that by the Markov property is not hard to show that for all 
$x \in E$ we have ${Var}^{\mathbb{P}^x}((e^{-\alpha t}u(X_t)_{t \geq 0}) = V^\alpha(u)(x)$, 
meaning that assertion i) holds if and only if $V^\alpha(u)<\infty$. But $V^\alpha(u)$ 
is a supremum of measurable functions taken over an uncountable set of partitions, hence
it may no longer be measurable, which makes it hard to handle in practice. 
Concerning this measurability issue, 
Theorem \ref{thm 2.1}, ii) states that instead of dealing with $V^\alpha(u)$, we can work with 
 $\mathop{\sup}\limits_{n}V^\alpha_{\tau_n}(u)$ for any admissible sequence of partitions $(\tau_n)_{n\geq 1}$. 
This subtile aspect was crucial in order to give criteria to check the quasimartingale nature of $u(X)$; see also Proposition \ref{prop 1} in the next subsection.    
\end{remark}

\subsection{Criteria for quasimartingale functions}   
\label{subsec:2.1}

In this subsection, still following \cite{BeCi16}, 
we provide general conditions for $u$ under which $(e^{-\beta t}u(X_t))_{t\geq 0}$
is a quasimartingale, which means that,  in particular, $(u(X_t))_{t\geq 0}$ is a semimartingale.  

\vspace{0.3cm}

Let us consider that $m$ is a $\sigma$-finite sub-invariant measure for $(P_t)_{t 
\geq 0}$ so that $(P_t)_{t \geq 0}$ 
extends uniquely to a strongly continuous semigroup of contractions on $L^p(m)$, $1 \leq p < \infty$; $\mathcal{U}$ may as well be extended to a strongly continuous resolvent family of contractions on $L^p(m)$, $1 \leq p < \infty$. 
The corresponding generators $({\sf L}_p, D({\sf L}_p) \subset L^p(m))$ are defined by
$$
D({\sf L}_p) = \{ U_{\alpha} f : f \in L^p(m) \},
$$
$$
{\sf L}_p(U_{\alpha} f) := \alpha U_{\alpha} f - f \quad {\rm for \; all} \; f \in L^p(m), \ 1 \leq p < \infty,
$$
with the remark that this definition is independent of $\alpha > 0$.

The corresponding notations for the dual structure are 
$\widehat{P}_t$ and $(\widehat{\sf L}_p, D(\widehat{\sf L}_p))$, 
and note that the adjoint of ${\sf L}_p$ is $\widehat{\sf L}_{p^\ast}$; $\frac{1}{p} + \frac{1}{p^\ast}=1$. Throughout, we denote the standard $L^p$-norms by $\| \cdot \|_p$, $1 \leq p \leq \infty$.

\vspace{0.2cm}

We present below the $L^p$-version of Theorem \ref{thm 2.1}; cf. Proposition 4.2 from \cite{BeCi16}.

\begin{proposition} \label{prop 1}   
The following assertions are equivalent for a $\mathcal{B}$-measurable function 
$u \in \mathop{\bigcup}\limits_{1 \leq p \leq \infty} L^p(m)$ and $\beta\geq 0$.

\vspace{0.2cm}

i) There exists an $m$-version $\widetilde{u}$ of $u$ such that 
$(e^{-\beta t}\widetilde{u}(X_t))_{t\geq 0}$ is a $\mathbb{P}^x$-quasimartingale for $x \in E$ $m$-a.e.

\vspace{0.2cm}

ii) For an admissible sequence of partitions 
$(\tau_n)_{n \geq 1}$ of $\mathbb{R}_+$, $\mathop{\sup}\limits_{n} V^\beta_{\tau_n}(u) < \infty$ $m$-a.e.

\vspace{0.2cm}

iii) There exist $u_1, u_2 \in E(\mathcal{U}_\beta)$ finite $m$-a.e. such that $u = u_1 - u_2$ $m$-a.e.
\end{proposition}

\begin{remark} \label{rem 4.3} Under the assumptions of Proposition \ref{prop 1}, 
if $u$ is finely continuous and one of the equivalent assertions is satisfied then 
all of the statements hold outside an $m$-polar set, not only $m$-a.e., 
since it is known that an $m$-negligible finely open set is automatically $m$-polar; 
if in addition $m$ is a reference measure then the assertions hold everywhere on $E$.
\end{remark}

Now, we focus our attention on a class of $\beta$-quasimartingale 
functions which arises as a natural extension of $D({\sf L}_p)$. 
First of all, it is clear that any function $u \in D({\sf L}_p)$, $1 \leq p < \infty$, 
has a representation 
$u = U_{\beta} f = U_{\beta}(f^+) - U_{\beta}(f^-)$ with $U_{\beta}(f^{\pm}) \in 
E(\mathcal{U}_{\beta}) \cap L^p(m)$, hence 
 $u$ has a $\beta$-quasimartingale version for all $\beta > 0$;
moreover, $\| P_t u - u \|_p = \left\| \int_0^t P_s{\sf L}_p u ds \right\|_p \leq t \| {\sf L}_p u \|_p$. 
The converse is also true, namely if $1 < p < \infty$, 
$u \in L^p(m)$, and $\| P_t u - u \|_p \leq {const} \cdot t$, $t \geq 0$, then $u \in D({\sf L}_p)$.  
But this is no longer the case if $p = 1$ (because of the lack of reflexivity of $L^1$), i.e.  
$\| P_t u - u \|_1 \leq {const} \cdot t$ does not imply $u \in D({\sf L}_1)$. 
However, it turns out that this last condition on $L^1(m)$ is yet enough to ensure that 
$u$ is a $\beta$-quasimartingale function.
In fact, the following general result holds; see \cite{BeCi16}, Proposition 4.4 and its proof.

\begin{proposition} \label{prop 3}  
Let $1 \leq p < \infty$ and suppose $\mathcal{A} \subset \{ u \in L^{p^{\ast}}_+(m) : \| u \|_{p^{\ast}} \leq 1 \}$, $\widehat{P}_s\mathcal{A} \subset \mathcal{A}$ for all $s \geq 0$, and 
$E = \mathop{\bigcup}\limits_{f \in \mathcal{A}}{\rm supp}(f)$ $m$-a.e. 
If $u \in L^p(m)$ satisfies

 \centerline{$\sup\limits_{f \in \mathcal{A}}\int_E |P_tu - u| f d m \leq {const} \cdot t$ for all $t \geq 0$,} 
\noindent
then there exists and $m$-version $\widetilde{u}$ of $u$ 
such that $(e^{-\beta t}\widetilde{u}(X_t))_{t\geq 0}$ 
is a $\mathbb{P}^x$-quasimartingale for all $x \in E$ $m$-a.e. and every $\beta  > 0$.
\end{proposition}

\vspace{0.2cm}

We end this subsection with the following criteria which is not given with respect to a duality measure, 
but in terms of the associated resolvent $\mathcal{U}$; cf. Proposition 4.1 from \cite{BeCi16}.

\begin{proposition} \label{prop 4}  
 Let $u$ be a real-valued $\mathcal{B}$-measurable finely continuous function. 

i) Assume there exist a constant $\alpha \geq 0$ and a non-negative $\mathcal{B}$-measurable function $c$ such that
$$
U_{\alpha}(|u| + c) < \infty, \quad \mathop{\lim\sup}\limits_{t \to \infty} P_t^{\alpha}|u| < \infty, \quad |P_t u - u| \leq c t, t \geq 0,
$$
and the functions $t \mapsto P_t(|u| + c)(x)$ are Riemann integrable.  
Then $(e^{-\alpha t}u(X_t))_{t \geq 0}$ is a $\mathbb{P}^x$-quasimartingale for all $x\in E$.

ii)  Assume there exist a constant $\alpha \geq 0$ and a non-negative $\mathcal{B}$-measurable function $c$ such that
$$
|P_t u- u| \leq c t, t \geq 0, \quad \mathop{\sup}\limits_{t \in \mathbb{R}_+} P_t^{\alpha}(|u| + c) = : b < \infty.
$$ 
Then $(e^{-\beta t}u(X_t))_{t \geq 0}$ is a $\mathbb{P}^x$-quasimartingale for all $x\in E$ and $\beta > \alpha$.

iii) Assume there exists $x_0 \in E$ such that for some $\alpha \geq 0$
$$
U_{\alpha}(|u|)(x_0) < \infty, \quad U_{\alpha}(|P_t u - u|)(x_0) \leq {const} \cdot t, \; t \geq 0.
$$
Then $(e^{-\beta t}u(X_t))_{t \geq 0}$ is a $\mathbb{P}^x$-quasimartingale for 
$\delta_{x_0}\circ U_\beta$-a.e. $x\in E$ and $\beta > \alpha$; 
if in addition $\mathcal{U}$ is strong Feller and topologically irreducible then 
the $\mathbb{P}^x$-quasimartingale property holds for all $x\in E$.
\end{proposition}

\subsection{Applications to semi-Dirichlet forms}  

Assume now that the semigroup $(P_t)_{t\geq 0}$ is associated 
to a semi-Dirichlet form $(\mathcal{E},\mathcal{F})$ on $L^2(E,m)$, 
where $m$ is a $\sigma$-finite measure on  the Lusin measurable space $(E,\cb)$;
as standard references 
for the theory of (semi-)Dirichlet forms we refer the reader to 
\cite{MaRo92}, \cite{MaOvRo95}, \cite{FuOsTa11}, \cite{Os13}, 
but also \cite{BeBo04}, Chapter 7. 
By Corollary 3.4 from \cite{BeBoRo06a}
there exists a (larger) Lusin topological space $E_1$ such
that  $E\subset E_1$, $E$ belongs to $\cb_1$ (the $\sigma$-algebra
of all Borel subsets of $E_1$), $\cb=\cb_1|_E$, 
and $(\mathcal{E}, \mathcal{F})$
regarded as a semi-Dirichlet form on $L^2(E_1 , \overline{m})$
is quasi-regular, where $\overline{m}$ is the trivial extension of $m$ to $(E_1,
\cb_1)$. 
Consequently, we may consider  a right Markov process $X$ with state space $E_1$ which 
is associated with the semi-Dirichlet form $(\mathcal{E},\mathcal{F})$.

If $u \in \mathcal{F}$ then $\widetilde{u}$ 
denotes a quasi continuous version of $u$ as a function on $E_1$ which always exists 
and it is uniquely determined quasi everywhere.
Following \cite{Fu99}, 
for a closed set $F$ we define $\mathcal{F}_{b, F}
:=\{v\in \mathcal{F} :  v \mbox{ is bounded and } v=0 \; m\mbox{-a.e. on } E\setminus F\}$.

The next result is a version of Theorem 5.5 from \cite{BeCi16}, 
dropping the a priori assumption that the semi-Dirichlet form is quasi-regular. 

\begin{theorem} \label{thm 3}
Let $u \in \mathcal{F}$ and assume there exist a nest $(F_n)_{n\geq1}$ and constants $(c_n)_{n\geq 1}$ such that
$$
\mathcal{E}(u,v) \leq c_n \|v\|_\infty \;\; \mbox{for all} \; v\in \mathcal{F}_{b, F_n}.
$$
Then $\widetilde{u}(X)$ is a $\mathbb{P}^x$-semimartingale for $x\in E_1$ quasi everywhere.
\end{theorem}

\begin{remark} 
The previous result has quite a history behind and 
we take the opportunity to recall some previous achievements on the subject. 
First of all, without going into details, note that if $E$ is a bounded domain in $\mathbb{R}^{d}$ 
(or more generally in an abstract Wiener space) and the condition from Theorem \ref{thm 3} 
holds for $u$ replaced by the canonical projections, 
then the conclusion is that the underlying Markov process is a semimartingale. 
In particular, the semimartingale nature of reflected diffusions on general bounded domains can be studied. 
This problem dates back to the work of \cite{BaHs90}, 
where the authors showed that the reflected Brownian motion on a Lipschitz domain in $\mathbb{R}^d$ 
is a semimartingale. Later on, this result has been extended to more general domains and diffusions; 
see \cite{WiZh90}, \cite{Ch93}, \cite{ChFiWi93}, and \cite{PaWi94}. 
A clarifying result has been obtained in \cite{ChFiWi93}, 
showing that the stationary reflecting Brownian motion on a bounded Euclidian domain 
is a quasimartingale on each compact time interval if and only if the domain is a strong Caccioppoli set. 
At this point it is worth to emphasize that in the previous sections we studied quasimartingales 
on the hole positive real semi-axis, not on finite intervals. 
This slight difference is a crucial one which makes our approach possible and completely different. 
A complete study of these problems (including Theorem \ref{thm 3} 
but only in the symmetric case) have been done in a series of papers by M. Fukushima and co-authors 
(we mention just \cite{Fu99}, \cite{Fu00}, and \cite{FuHi01}), 
with deep applications to BV functions in both finite and infinite dimensions. 

All these previous results have been obtained using the same common tools: symmetric Dirichlet forms and Fukushima decomposition. 
Further applications to the reflection problem in infinite dimensions have been studied in \cite{RoZhu12} and \cite{RoZhu15}, 
where non-symmetric situations were also considered. 
In the case of semi-Dirichlet forms, a Fukushima decomposition is not yet known to hold, 
unless some additional hypotheses are assumed (see e.g. \cite{Os13}). 
Here is where our study developed in the previous sections played its role, 
allowing us to completely avoid Fukushima decomposition or the existence of the dual process. 
On brief, the idea of proving Theorem \ref{thm 3} is to show that locally, 
the conditions from Proposition \ref{prop 3} are satisfied, so that $u(X)$ is (pre)locally a semimartingale, and hence a global semimartingale.
\end{remark}

Assume  that $(\mathcal{E}, \mathcal{F})$ is  quasi-regular and that it is  
{\it local}, i.e., $\mathcal{E}(u,v)=0$ 
for all $u,v \in \mathcal{F}$ with disjoint compact supports. 
It is well known that the local property is equivalent with the fact that the associated process is a diffusion; 
see e.g. \cite{MaRo92}, Chapter V, Theorem 1.5. As in \cite{Fu00}, 
the local property of $\mathcal{E}$ allows us to extend Theorem \ref{thm 3} 
to the case when $u$ is only locally in the domain of the form, or to even more general situations, as stated in the next result; 
for details see Subsection 5.1  from \cite{BeCi16}.

\begin{corollary} \label{coro 5.4}  
Assume that $(\mathcal{E}, \mathcal{F})$ is local.
Let $u$ be a real-valued $\mathcal{B}$-measurable finely continuous function and let 
$(v_k)_k \subset \mathcal{F}$ such that $v_k \mathop{\longrightarrow}\limits_{k \to\infty} u$ 
point-wise outside an $m$-polar set and boundedly on each element of a nest $(F_n)_{n \geq 1}$. 
Further, suppose that there exist constants $c_n$ such that
$$
|\mathcal{E}(v_k, v)| \leq c_n \|v\|_{\infty} \;\; for \; all \; v \in \mathcal{F}_{b, F_n}.
$$
Then $u(X)$ is a $\mathbb{P}^x$-semimartingale for $x \in E$ quasi everywhere.
\end{corollary}

\section{Excessive and invariant functions on $L^p$-spaces}  \label{sect.3}    

Throughout this section $\mathcal{U}=(U_\alpha)_{\alpha > 0}$ is a sub-Markovian resolvent of kernels on 
$E$ and $m$ is a $\sigma$-finite sub-invariant measure, i.e. $m( \alpha U_\alpha f ) \leq m(f)$ for all $\alpha >0$ and non-negative $\mathcal{B}$-measurable functions $f$; 
then there exists a second sub-Markovian resolvent of kernels on $E$ denoted by 
$\mathcal{\widehat{U}}=(\widehat{U}_\alpha)_{\alpha>0}$ 
which is in {\it weak duality}  with $\mathcal{U}$  w.r.t. $M$ in the sense that 
$\int_E fU_\alpha g dm = \int_E g\widehat{U}_\alpha f dm$ 
for all positive $\mathcal{B}$-measurable functions $f,g$ and $\alpha >0$. 
Moreover, both resolvents can be extended to contractions on any $L^p(E,m)$-space for all $1\leq p\leq \infty$, 
and if they are strongly continuous then we keep the same notations for their generators as in Subsection \ref{subsec:2.1}. 
In this part, our attention focuses on a special class of differences of excessive functions 
(which are in fact harmonic when the resolvent is Markovian). Extending \cite{AlKoRo97a}, they are defined as follows.

\begin{definition} A real-valued $\mathcal{B}$-measurable function 
$v \in \bigcup_{1 \leq p \leq \infty} L^p(E, m)$ is called {\it $\mathcal{U}$-invariant} provided that 
$U_{\alpha}(vf) = v U_{\alpha} f$ $m$-a.e. for all bounded and $\mathcal{B}$-measurable functions 
$f$ and $\alpha > 0$. 
\end{definition}

A set $A \in \mathcal{B}$ is called {\it $\mathcal{U}$-invariant} if $1_A$ is $\mathcal{U}$-invariant; 
the collection of all $\mathcal{U}$-invariant sets is a $\sigma$-algebra.

\begin{remark} \label{rem 2.14} 
If $v \geq 0$ is $\mathcal{U}$-invariant, then by \cite{BeCiRo15}, 
Proposition 2.4 there exists $u \in E(\mathcal{U})$ such that $u = v$ $m$-a.e. 
If $\alpha U_{\alpha} 1 = 1$ $m$-a.e. then for every invariant function $v$ 
we have that $\alpha U_{\alpha} v = v$ $m$-a.e, which is equivalent (if $\mathcal{U}$ 
is strongly continuous) with $v$ being ${\sf L}_p$-harmonic, i.e. $v \in D({\sf L}_p)$ and ${\sf L}_pv=0$. 
\end{remark}

The following result is a straightforward consequence of the duality between $\mathcal{U}$ and $\widehat{\mathcal{U}}$;
for  its proof see Proposition 2.24 and Proposition 2.25 from \cite{BeCiRo15}.

\begin{proposition} \label{prop 2.15}
The following assertions hold.

i)  A function $u$ is $\mathcal{U}$-invariant if and only if it is $\widehat{\mathcal{U}}$-invariant.

ii) The set of all $\mathcal{U}$-invariant functions from $L^p(E, m)$
 is a vector lattice with respect to the point-wise order relation.
\end{proposition}

Let
\[
\mathcal{I}_p : = \{ u \in L^p(E, m) : \alpha U_\alpha u = u \; m\mbox{-a.e.}, \; \alpha > 0 \}.  
\]

The main result here is the next one, and it unifies and extends different more or less known characterizations of invariant functions;
cf. Theorem 2.27 and Proposition 2.29 from \cite{BeCiRo15}.

\begin{theorem} \label{thm 2.19} 
Let $u \in L^p(E, m)$, $1 \leq p < \infty $, and consider the following conditions.

\vspace{0.1cm}

i) $\alpha U_{\alpha}u = u$ $m$-a.e. for one (and therefore for all) $\alpha > 0$.

\vspace{0.2cm}

ii) $\alpha \widehat{U}_{\alpha} u = u$ $m$-a.e., $\alpha > 0$. 

\vspace{0.1cm}

iii) The function $u$ is $\mathcal{U}$-invariant.

\vspace{0.1cm}

iv) $U_\alpha u = u U_\alpha 1$ and $ \widehat{U}_{\alpha} u = u \widehat{U}_{\alpha} 1$ $m$-a.e. for one (and therefore for all) $\alpha > 0$.

\vspace{0.1cm}

v) The function $u$ is measurable w.r.t. the $\sigma$-algebra of all $\mathcal{U}$-invariant sets.

\vspace{0.1cm}

Then $\mathcal{I}_p$ is a vector lattice w.r.t. the pointwise order relation and i) $\Leftrightarrow$ ii) $\Rightarrow$ iii) $\Leftrightarrow$ iv) $\Leftrightarrow$ v).

If $\alpha U_\alpha 1 = 1$ or $\alpha \widehat{U}_{\alpha} 1 = 1$ $m$-a.e. then assertions i) - v) are equivalent.

If $p= \infty$ and $\mathcal{U}$ is $m$-recurrent (i.e. there exists $0\leq f \in L^1(E,m)$ s.t. $Uf=\infty$ $m$-a.e.) 
then the assertions i)-v) are equivalent.
\end{theorem}

\begin{remark} \label{rem 2.16} 
Similar characterizations for invariance as in Theorem \ref{thm 2.19}, but in the recurrent case and for functions which are bounded or integrable with bounded negative parts were already investigated in \cite{Sc04}. Of special interest is the situation when the only invariant functions are the constant ones ({\it irreducibility}) because it entails ergodic properties for the semigroup resp. resolvent; see e.g. \cite{St94}, \cite{AlKoRo97a}, and \cite{BeCiRo15}.
\end{remark}


\subsection{Martingale functions with respect to the dual Markov process}  

Our aim in this subsection is to identify the $\mathcal{U}$-invariant functions with martingale functions
and  co-martingale ones (i.e., martingales w.r.t some dual process); cf. Corollary \ref{cor3} below.
The convenient frame is that from \cite{BeRo15} and we present it here briefly.

Assume that  $\cu=(U_{\alpha})_{\alpha>0}$ is the resolvent of a right process $X$ with state space $E$ and
let ${\mathcal T}_0$ be the  Lusin topology  of $E$
having $\cb$ as Borel $\sigma$-algebra, and let $m$ be a fixed $\cu$-excessive measure.
Then   by Corollary 2.4 from \cite{BeRo15}, and using also the result from \cite{BeBoRo06a},
the following assertions hold:
{\it There exist a larger Lusin measurable space $(\overline E, \overline\cb)$, 
with $E\subset \overline E$, $E\in \overline{\cb}$, $\cb=\overline\cb |_{E}$,
and two processes $\overline X$ and $\widehat{{X}}$ with common state space $\overline E$, 
such that $\overline X$ 
is a right process on $\overline E$ endowed with a convenient Lusin  topology 
having $\overline\cb$ as Borel $\sigma$-algebra
(resp.  $\widehat{{X}}$ is a right process w.r.t. to a second Lusin topology on $\overline E$ , also generating $\overline\cb$), 
the restriction of $\overline X$
to $E$ is precisely $X$, and the resolvents of $\overline X$ and $\widehat{{X}}$ 
are in duality with respect to $\overline m$, where $\overline{m}$ is the trivial extension of $m$ to $(E_1,\cb_1): \; \overline{m}(A):=m(A\cap E), \; A \in \mathcal{B}_1$.
In addition,   the $\alpha$-excessive functions, $\alpha>0$, 
with respect to $\widehat{X}$ on $\overline E$  are precisely the unique extensions by continuity 
in the fine topology generated by $\widehat{X}$ of the
$\widehat{\cu}_{\alpha}$-excessive functions. 
In particular,  the set $E$ is dense in $\overline E$ in the fine topology of $\widehat{X}$.
}

Note that the strongly continuous resolvent of sub-Markovian contractions induced on $L^p(m)$, $1\leq p<\infty$, 
by the process
$\overline X$  (resp.  $\widehat{{X}}$)  coincides  with $\cu$ (resp. $\widehat\cu$).


\begin{corollary} \label{cor2}
Let $u$ be function from  $L^p(E, m)$, $1 \leq p < \infty $.
Then the  following assertions are equivalent.

i) The process $(u(X_t))_{t\geq 0}$ is a martingale w.r.t. $\mathbb{P}^x$ for $m$-a.e. $x\in E$.

ii) The process $(u(\widehat{X}_t))_{t\geq 0}$ is a martingale w.r.t. $\widehat{\mathbb{P}}^x$ for $m$-a.e. $x\in E$.

iii) The function $u$ is ${\sf L}_p$-harmonic, i.e. $u \in D({\sf L}_p)$ and ${\sf L}_p u=0$. 

iv)  The function $u$ is ${\widehat{\sf L}}_p$-harmonic, i.e. $u \in D({\widehat{\sf L}}_p)$ and ${\widehat{\sf L}}_p u=0$. 
\end{corollary}

\begin{proof}
The equivalence $iii) \Longleftrightarrow iv)$ follows by Theorem \ref{thm 2.19}, $i) \Longleftrightarrow ii)$,
while the equivalence $i) \Longleftrightarrow iii)$ is a consequence of Proposition \ref{prop1}.
$\hfill\square$
\end{proof}

We make the transition to the next (also the last) section of this paper with an application of 
Theorem \ref{thm 2.19} to the existence of invariant probability measures for Markov processes. 
More precisely, assume that $\mathcal{U}$ is the resolvent of a right Markov process with transition function $(P_t)_{t \geq 0}$. 
As before, $m$ is a $\sigma$-finite sub-invariant measure for $\mathcal{U}$ 
(and hence for $(P_t)_{t\geq 0}$), while ${\sf L}_1$ and $\widehat{{\sf L}}_1$ stand for the generator, resp. the co-generator on $L^1(E,m)$. 

\begin{corollary} \label{cor3} The following assertions are equivalent.

i) There exists an invariant probability measure for $(P_t)_{t\geq 0}$ which is absolutely continuous w.r.t. $m$.

ii) There exists a non-zero element $\rho \in D({\sf L}_1)$ such that ${\sf L}_1 \rho = 0$.

\end{corollary}

\begin{proof}
It is well known that a probability measure $\rho \cdot m$ is invariant w.r.t. $(P_t)_{t \geq 0}$ 
is equivalent with the fact that $\rho \in D(\widehat{\sf L}_1)$ and $\widehat{{\sf L}}_1 \rho=0$ 
(see also Lemma \ref {lem 2.1}, ii) from below). Now, the result follows by Theorem \ref{thm 2.19}.  
\end{proof}

\begin{remark} Regarding the previous result, we point out that if 
$m(E) < \infty$ and $(P_t)_{t \geq 0}$ is conservative 
(i.e. $P_t1=1$ $m$-a.e. for all $t>0$) 
then it is clear  that $m$ itself is invariant, so that Corollary \ref{cor3} 
has got a point only when $m(E)=\infty$. 
Also, we emphasize that the sub-invariance property of $m$ is an essential assumption. 
We present a general result on the existence of invariant probability measures in the next section, 
where we drop the sub-invariance hypothesis.
\end{remark}

\section{$L^1$-harmonic functions and invariant probability measures}   \label{sec:4}  

Throughout this subsection $(P_t)_{t \geq 0}$ is a measurable 
Markovian transition function on a measurable space $(E, \mathcal{B})$ and $m$ is an 
{\it auxiliary} measure for $(P_t)_{t\geq 0}$, i.e. a finite positive measure such that 
$m(f) = 0 \Rightarrow m(P_tf) = 0$ for all $t > 0$ and all positive $\mathcal{B}$-measurable functions $f$. 
As we previously announced, our final interest concerns the existence of an invariant probability measure for 
$(P_t)_{t \geq 0}$ which is absolutely continuous with respect to $m$.

\begin{remark}
We emphasize once again that in contrast with the previous section, $m$ is not assumed sub-invariant, 
since otherwise it would be automatically invariant. Also, any invariant measure is clearly auxiliary, but the converse is far from being true. 
As a matter of fact, the condition on $m$ of being auxiliary is a minimal one: 
for every finite positive measure $\mu$ and $\alpha >0$ one has that $\mu \circ U_\alpha$ is auxiliary; see e.g. \cite{RoTr07} and \cite{BeCiRo15a}. 
\end{remark}

For the first assertion of the next result we refer to \cite{BeCiRo15a}, Lemma 2.1, while the second one is a simple consequence of the fact that $P_t1=1$.

\begin{lemma} \label{lem 2.1}
i) The adjoint semigroup $(P_t^{\ast})_{t \geq 0}$ on $(L^{\infty}(m))^{\ast}$ maps $L^1(m)$ into itself, 
and restricted to $L^1(m)$ it becomes a semigroup of positivity preserving operators.

ii) A probability measure $\rho\cdot m$ is invariant with respect to $(P_t)_{t \geq 0}$ 
if and only if $\rho $ is $m$-co-excessive, i.e. $P_t^\ast \rho \leq \rho$ $m$-a.e. for all $t \geq 0$.
\end{lemma}

Inspired by well known ergodic properties for semigroups and resolvents 
(see for example \cite{BeCiRo15}), our idea in order to produce co-excessive functions 
is to apply (not for $(P_t)_{t \geq 0}$ but for its adjoint semigroup) 
a compactness result in $L^1(m)$ due to \cite{Ko67}, saying that an 
$L^1(m)$-bounded sequence of elements possesses a subsequence 
whose Cesaro means are almost surely convergent to a limit from $L^1(m)$.

\begin{definition}
The auxilliary measure $m$ is called {\it almost invariant} for $(P_t)_{t \geq 0}$ 
if there exist $\delta \in [0, 1)$ and a set function $\phi : \mathcal{B} \rightarrow \mathbb{R}_+$ 
which is absolutely continuous with respect to $m$ (i.e. $\mathop{\lim}\limits_{m(A) \to 0} \phi(A) = 0$) such that 

$$m (P_t1_A) \leq \delta m(E) + \phi(A) \quad \mbox{for all} \; t > 0. $$
\end{definition}

Clearly, any positive finite invariant measure is almost invariant. 
Here is our last main result,  a variant of Theorem 2.4 from  \cite{BeCiRo15a}.

\begin{theorem} \label{thm 2.3}

The following assertions are equivalent.

i) There exists a nonzero positive finite invariant measure for $(P_t)_{t \geq 0}$ which is absolutely continuous with respect to $m$.

\vspace{0.2cm}

ii) $m$ is almost invariant.

\end{theorem}

\begin{acknowledgement}
The first named author acknowledges support from the Romanian National Authority for Scientific Research, 
project number PN-III-P4-ID-PCE-2016-0372. 
The second named author acknowledges support from the Romanian National Authority for Scientific Research, project number PN-II-RU-TE-2014-4-0657.
\end{acknowledgement}

\section*{Appendix}

\noindent
{\bf Proof of Proposition \ref{prop1}.}

i) $\Rightarrow$ ii). If $(e^{-\beta t}u(X_t))_{t\geq 0}$ is a right-continuous supermartingale 
then by taking expectations we get that $e^{-\beta t} \mathbb{E}^x u(X_t) \leq \mathbb{E}^xu(X_0)$, hence $u$ is $\beta$-supermedian. 
Now, by \cite{BeBo04}, Corollary 1.3.4, showing that $u \in E(\mathcal{U}_\beta)$ reduces to prove that $u$ is finely continuous, 
which in turns follows by the well known characterization  according to which $u$ is finely continuous if and only if $u(X)$ 
has right continuous trajectories $\mathbb{P}^x$-a.s. for all $x \in E$; 
see Theorem 4.8 in \cite{BlGe68}, Chapter II.

ii) $\Rightarrow$ i). Since $u$ is $\beta$-supermedian, by the Markov property we have for all $0\leq s \leq t$
$$
\mathbb{E}^x[e^{-\beta (t+s)}u(X_{t+s}) | \mathcal{F}_s]=e^{-\beta (t+s)}\mathbb{E}^{X_s}u(X_t)=e^{-\beta (t+s)} P_tu(X_s) \leq e^{-\beta s}u(X_s),
$$
hence $(e^{-\beta t}u(X_t))_{t\geq 0}$ is an $\mathcal{F}_t$-supermartingale. 
The right-continuity of the trajectories follows by the fine continuity of $u$ via the previously mentioned characterization. 
$\hfill\square$

%
%
%

\end{document}